\documentclass[oneside,english]{amsart}
\usepackage[T1]{fontenc}
\usepackage[latin9]{inputenc}
\usepackage{color}
\usepackage{babel}
\usepackage{amsthm}
\usepackage{amssymb}
\usepackage{esint}
\usepackage[unicode=true,pdfusetitle,
 bookmarks=true,bookmarksnumbered=false,bookmarksopen=false,
 breaklinks=false,pdfborder={0 0 0},backref=false,colorlinks=true]
 {hyperref}

\makeatletter
\numberwithin{equation}{section}
\numberwithin{figure}{section}
\theoremstyle{plain}
\newtheorem{thm}{\protect\theoremname}
  \theoremstyle{plain}
  \newtheorem{lem}[thm]{\protect\lemmaname}
  \theoremstyle{definition}
  \newtheorem{defn}[thm]{\protect\definitionname}
  \theoremstyle{remark}
  \newtheorem*{rem*}{\protect\remarkname}
  \theoremstyle{plain}
  \newtheorem{prop}[thm]{\protect\propositionname}
  \theoremstyle{plain}
  \newtheorem{cor}[thm]{\protect\corollaryname}

\makeatother

  \providecommand{\corollaryname}{Corollary}
  \providecommand{\definitionname}{Definition}
  \providecommand{\lemmaname}{Lemma}
  \providecommand{\propositionname}{Proposition}
  \providecommand{\remarkname}{Remark}
\providecommand{\theoremname}{Theorem}

\begin{document}

\title[$q$-heat flow and the gradient flow of the Renyi entropy]{$q$-heat flow and the gradient flow of the Renyi entropy in the
$p$-Wasserstein space}

\author{Martin Kell}

\email{mkell@mis.mpg.de}

\address{Max-Planck-Institute for Mathematics in the Sciences, Inselstr. 22,
04103 Leipzig, Germany}

\thanks{The author wants to thank Prof. Jürgen Jost and the MPI MiS for providing
a stimulating research environment. The research was funded by the
IMPRS ``Mathematics in the Sciences''.}
\begin{abstract}
Based on the idea of a recent paper by Ambrosio-Gigli-Savar\'e in
Invent. Math. (2013), we show that flow of the $q$-Cheeger energy,
called $q$-heat flow, solves the gradient flow problem of the Renyi
entropy functional in the $p$-Wasserstein. For that, a further study
of the $q$-heat flow is presented including a condition for its mass
preservation. Under a convexity assumption on the upper gradient,
which holds for all $q\ge2$, one gets uniqueness of the gradient
flow and the two flows can be identified. Smooth solution of the $q$-heat
flow are solution the parabolic $q$-Laplace equation, i.e. $\partial_{t}f_{t}=\Delta_{q}f_{t}$. 
\end{abstract}
\maketitle
The heat flow induced by a Dirichlet form is by now a well-understood
concept. In \cite{JKO1998} Jordan, Kinderlehrer and Otto showed in
the Euclidean setting that one can identify the heat flow with the
gradient flow of the entropy functional in the $2$-Wasserstein space.
The main idea was to show that the solution of the gradient flow problem
solves also the heat equation. Uniqueness of the solution implies
that the two flows are identical. The identification of the heat flow
and the gradient flow of the entropy functional on manifolds was later
accomplished by Erbar \cite{Erbar2010}. 

Otto \cite{Otto1996,Otto2001} also gave a formal proof of how to
use gradient flows in the $p$-Wasserstein spaces modeled on $\mathbb{R}^{n}$
in order to solve other equations like the porous media equation and
the parabolic $q$-Laplace equation, i.e. the $q$-heat flow. Rigorous
proofs were later given by Agueh \cite{Agueh2002,Agueh2005}. Only
recently Ohta and Takatsu \cite{OT2011,OT2011a} also showed that
a similar construction works on manifolds if the functionals are $K$-convex. 

All proofs until then required the contraction property which follows,
at least in the Riemannian setting, from the curvature dimension condition
introduced by Lott-Villani and Sturm \cite{LV2009,LV2007,Sturm2006a,Sturm2006}.
Since this condition can be defined on any metric measure spaces it
was believed that a similar identification holds also under such a
condition. In \cite{Gigli2009} Gigli gave a proof which did not require
the contraction property. This proof let Ambrosio-Gigli-Savar\'e
\cite{Ambrosio2013} to define a new generalized gradient from which
one gets a natural heat flow associated to a metric space. With the
help of a calculus of the heat flow and its mass preservation they
could show that the heat flow is a solution of the gradient flow problem
of the entropy functional in the $2$-Wasserstein space. Using a convexity
of the square of the upper gradient of the entropy functional one
gets uniqueness and hence the two flows are identical. 

One of the main ingredient of the proof was the Kuwada lemma, i.e.
if $\mu_{t}=f_{t}\mu$ is a solution of the heat flow and $|\dot{\mu}_{t}|$
is the metric derivative of $t\mapsto\mu_{t}$ in the $2$-Wasserstein
space $\mathcal{P}_{2}(M)$ then 
\[
|\dot{\mu}_{t}|^{2}\le\int\frac{|\nabla f_{t}|^{2}}{f_{t}}d\mu
\]
where the write hand side is called the Fisher information of $f_{t}$.
This was the ``missing'' ingredient, since it was long known that
the derivative along the heat flow $t\mapsto f_{t}$ of the entropy
functional is (minus) the Fisher information of $f_{t}$. 

In \cite{Ambrosio2011} Ambrosio-Gigli-Savar\'e showed the Kuwada
lemma for $q\ne2$, namely if $t\mapsto f_{t}$ is the $q$-heat flow
such that the density is bounded from above and away from zero from
below (implying the measure $\mu$ is finite), they showed 
\[
|\dot{\mu}_{t}|^{p}\le\int\frac{|\nabla f_{t}|^{q}}{f_{t}^{p-1}}d\mu
\]
where this time the metric derivative is taken in the $p$-Wasserstein
space $\mathcal{P}_{p}(M)$, $t\mapsto\mu_{t}=f_{t}\mu$ is a solution
of the $q$-heat flow and $p$ and $q$ are Hölder conjugates. A formal
calculation reveals that the derivative of the following functional
\[
f\mapsto\frac{1}{(3-p)(2-p)}\int f^{3-p}-fd\mu,
\]
called $(3-p)$-Renyi entropy, along the $q$-heat flow in the $p$-Wasserstein
space is exactly minus the right hand side of the previous inequality,
which can be called the $q$-Fisher information. 

In this paper, we will follow \cite{Ambrosio2013} and first develop
a calculus of the $q$-heat flow to show mass preservation in the
non-compact setting and that the formal calculation above holds in
an abstract setting. In case $q>2$ there is almost no restriction
on the measure to get mass preservation besides a ``not too bad''
growth of the measure of a ball. The cases $q<2$ are more restrictive.
Using generalized exponential functions already know from information
theory \cite[Section 3]{OT2011} one of the conditions can be stated
as 
\[
\int\exp_{p}(-V^{p})d\mu<\infty
\]
where $V(x)=Cd(x,x_{0})$ for some $C>0$ and $exp_{p}$ is the generalized
exponential function which agrees with the usual exponential function
and the condition with the condition stated in \cite{Ambrosio2013}.
In $\mathbb{R}^{n}$ this condition boils down to $q>\frac{2n}{n+1}$.
However, the current proof requires the more restrictive condition
\[
\int V^{p}\exp_{p}(-V^{p})d\mu<\infty.
\]

In the second part under some assumptions on the functional, which
hold assuming a curvature condition defined in a previous paper \cite{Kell2013a},
we show that the proof of \cite{Ambrosio2013} can be adjusted to
show that the $q$-heat flow solves the gradient flow problem of the
Renyi entropy in $\mathcal{P}_{p}(M)$. For $q>2$ we also get convexity
of the $q$-the power of the upper gradient and hence uniqueness of
the gradient flow. This implies that the $q$-heat flow and the gradient
flow of the Renyi entropy can be identified. The current proof of
the cases $ $$q<2$ requires the space to be compact and the measure
be $n$-Ahlfors regular for some $n$ depending on $q$. However,
this condition is satisfied on smooth manifolds if the the curvature
condition $CD_{p}(0,N)$, defined in a previous paper \cite{Kell2013a},
holds for $N>n$.

\section{Preliminaries}

In this part, we will introduce the main concepts used in this work.
We will follow the notation used in \cite{Ambrosio2013}. For a general
introduction to the theory of optimal transport via $2$-Wasserstein
spaces see \cite{Villani2009}, especially its Chapter 6 on Wasserstein
spaces. 

Let $(X,d)$ be a (complete) metric space and for simplicity we assume
that $X$ has no isolated points. As a convention we will always assume
that $(M,d,\mu)$ is a locally compact metric space equipped with
a locally finite Borel measure $\mu$ and if not otherwise stated
it is assumed to be geodesic (see below). Since we will also deal
with spaces which are not locally compact (e.g. $(\mathcal{P}_{p}(M),w_{p})$
with $M$ non-compact), the sections below do not assume that $(X,d)$
is locally compact.

\subsection*{Lipschitz constants and upper gradients}

Given a function $f:X\to\overline{\mathbb{R}}=[-\infty,\infty]$,
the local Lipschitz constant $|Df|:X\to[0,\infty]$ is given by 
\[
|Df|(x):=\limsup_{y\to x}\frac{|f(y)-f(x)|}{d(y,x)}
\]
for $x\in D(f)=\{x\in X\,|\, f(x)\in\mathbb{R}\}$, otherwise $|Df|(x)=\infty$.
The one sided versions $|D^{+}f|$ and $|D^{-}f|$, also called ascending
slope (resp. descending slope)
\begin{eqnarray*}
|D^{+}f|(x) & := & \limsup_{y\to x}\frac{[f(y)-f(x)]_{+}}{d(y,x)}\\
|D^{-}f|(x) & := & \limsup_{y\to x}\frac{[f(y)-f(x)]_{-}}{d(y,x)}
\end{eqnarray*}
for $x\in D(f)$ and $\infty$ otherwise, where $[r]_{+}=\max\{0,r\}$
and $[r]_{-}=\max\{0,-r\}$. It is not difficult to see that $|Df|$
is (locally) bounded iff $f$ is (locally) Lipschitz. 

The following lemma will be crucial to calculate the derivative of
functionals along the gradient flow of the Cheeger energy.
\begin{lem}
[{\cite[Lemma 2.5]{Ambrosio2013}}] \label{lem:conv-diff}Let $f,g:X\to\mathbb{R}$
be (locally) Lipschitz functions, $\phi:\mathbb{R}\to\mathbb{R}$
be a $C^{1}$-function with $0\le\phi'\le1$ and $\psi:[0,\infty)\to\mathbb{R}$
be a convex nondecreasing function. Setting
\[
\tilde{f}:=f+\phi(g-f),\qquad\tilde{g}=g-\phi(g-f)
\]
we have for every $x\in X$
\[
\psi(|D\tilde{f}|)(x)+\psi(|D\tilde{g}|)(x)\le\psi(|Df|)(x)+\psi(|Dg|)(x).
\]

\end{lem}
We say that $g:X\to[0,\infty]$ is an upper gradient of $f:X\to\overline{\mathbb{R}}$
if for any absolutely continuous curve $\gamma:[0,1]\to D(f)$ the
curve $t\mapsto g(\gamma_{s})|\dot{\gamma}_{s}|$ is measurable in
$[0,1]$ (with convention $0\cdot\infty=0$) and 
\[
|f(\gamma_{1})-f(\gamma_{0})|\le\int_{0}^{1}g(\gamma_{t})dt.
\]
It is not difficult to see that the local Lipschitz constant and the
two slopes are upper gradients in case $f$ is (locally) Lipschitz.

\subsection*{Relaxed slope and the Cheeger energy}

In a metric space there is no natural gradient of $L^{r}$-functions
which are not Lipschitz. Cheeger defined in \cite{Cheeger1999} a
gradient via a relaxation procedure using slopes of Lipschitz function.
In \cite{Ambrosio2013,Ambrosio2011} Ambrosio-Gigli-Savar\'e used
a more restrictive version of Cheeger's original definition.
\begin{defn}
[$q$-relaxed slope] A function $g\in L^{q}$ is a $q$-relaxed slope
of $f\in L^{2}$ if there is a sequence of Lipschitz functions $f_{n}$
strongly converging to $f$ in $L^{2}$ such that $|Df_{n}|$ converges
weakly (in $L^{q}$) to some $\tilde{g}\in L^{q}$ with $g\le\tilde{g}$.
We denote by $|\nabla f|_{*,q}$ the element of minimal $L^{q}$-norm
among all $q$-relaxed slopes.\end{defn}
\begin{rem*}
In order to apply the gradient flow theory of Hilbert spaces, we divert
from the approach in \cite{Ambrosio2011} and use approximations of
$f$ in $L^{2}$ instead of $L^{q}$. Note that the proofs of \cite{Ambrosio2011}
also work in this setting if appropriate changes are made.
\end{rem*}
It was shown in \cite{Ambrosio2011} that this definition, Cheeger's
original and two other definitions agree almost everywhere. However,
if the space does not satisfy a local doubling condition and a local
Poincar\'e inequality, then the $q$-relaxed slope might be different
from the $q'$-relaxed slope if $q\ne q'$, see \cite{DiMarino2013}.
Nevertheless, we will drop the dependency on $q$ and just write $|\nabla f|_{*}$.

One can show that the relaxed slope is sublinear, i.e. $|\nabla(f+g)|_{*}\le|\nabla f|_{*}+|\nabla g|_{*}$
almost everywhere, and satisfies a weak form of the chain rule, i.e.
for any $C^{1}$-function $\phi:\mathbb{R}\to\mathbb{R}$, which is
Lipschitz on the image of $f$, we have $|\nabla\phi(f)|_{*}\le|\phi'(f)||\nabla f|_{*}$
with equality if $\phi$ is non-decreasing \cite[Proposition 4.8]{Ambrosio2013}.
This can be easily proven for Lipschitz functions and their slopes,
and follows by a cut-off argument also for functions and their relaxed
slopes. 

Now the $q$-Cheeger energy of the metric measure space $(M,d,\mu)$
is defined as
\[
\operatorname{Ch}_{q}(f)=\frac{1}{q}\int|\nabla f|_{*}^{q}d\mu
\]
for all $f$ admitting a relaxed slopes, otherwise $\operatorname{Ch}_{q}(f)=\infty$.
Similarly, given a convex increasing function $L:[0,\infty)\to[0,\infty)$
with $L(0)=0$, the $L$-Cheeger energy is defined as $\operatorname{Ch}_{L}(f)=\int L(|\nabla f|)d\mu$.
Then the $q$-Cheeger energy is nothing but the $L$-Cheeger energy
for $L(r)=r^{q}/q$.
\begin{prop}
\label{prop:grad-contr}Let $f,g\in D(\operatorname{Ch}_{q})$ and
$\phi:\mathbb{R}\to\mathbb{R}$ be a nondecreasing contraction (with
$\phi(0)=0$ if $\mu(M)=\infty$) then $\mu$-almost everywhere in
$M$ 
\[
|\nabla(f+\phi(g-f))|_{*}^{q}+|\nabla(g-\phi(g-f))|_{*}^{q}\le|\nabla(f)|_{*}^{q}+|\nabla(g)|_{*}^{q}.
\]
\end{prop}
\begin{proof}
The proof follows along the lines of the proof of \cite[Proposition 4.8]{Ambrosio2013}
using Lemma \ref{lem:conv-diff} (see \cite[Lemma 2.5]{Ambrosio2013}).
\end{proof}

\subsection*{Fisher information}

The Fisher information is the derivative of the entropy functional
along the heat flow. The Kuwada lemma, a key tool of \cite{Ambrosio2013}
to identify the heat flow and the gradient flow of the entropy functional,
shows that the square of the metric derivative in the $2$-Wasserstein
space along the heat flow is bounded from above by the Fisher information.
In a different paper \cite{Ambrosio2011} they showed that in the
compact setting with density of the measure bounded from below and
above, there is also a version of this along the $q$-heat flow in
the $p$-Wasserstein space (see Lemma \ref{lem:Kuwada} for a precise
version). For that reason we define the following $q$-Fisher information
as follows.
\begin{defn}
[$q$-Fisher information] Let $q\in(\frac{1+\sqrt{5}}{2},\infty)$.
For a Borel function $f:M\to[0,\infty]$ we define the $q$-Fisher
information $\mathsf{F}_{q}(f)$ as 
\[
\mathsf{F}_{q}(f):=r^{-q}\int|\nabla f^{r}|_{*}^{q}d\mu=qr^{-q}\operatorname{Ch}_{q}(f^{r})
\]
where $q\ne\frac{1+\sqrt{5}}{2}$ and 
\[
r=1-\frac{p-1}{q}=1-\frac{(p-1)^{2}}{p}.
\]

In case $q=\frac{1+\sqrt{5}}{2}$, note $q=p-1$ and thus we define
\[
\mathsf{F}_{q}(f)=\int|\nabla\log f|_{w}^{q}d\mu=q\operatorname{Ch}_{q}(\log f).
\]
\end{defn}
\begin{rem*}
For $q\in(\frac{1+\sqrt{5}}{2},\infty)$, we also have $r\in(0,1)$,
which will be our main interest for technical reasons. Nevertheless,
all case $q\ge2$ are covered. In the following, we will just write
$r>0$. Furthermore, notice that $N\ge2$ and $1-\frac{1}{N}=3-p$
implies $p=2+\frac{1}{N}\le2.5<\frac{3+\sqrt{5}}{2}$. Thus only the
cases $N\in(1,2)$ remain to be covered. In the smooth setting $CD_{p}(K,N)$
with $N\in(1,2)$ can only hold for $1$-dimensional spaces.\end{rem*}
\begin{prop}
\label{prop:Fisher}Let $r>0$. Then for every Borel function $f:M\to[0,\infty]$
we have the equivalence
\[
f\in D(\mathsf{F}_{q})\;\Longleftrightarrow\; f\in L^{2r}(M,\mu)\:\mbox{ and }\:\int_{\{f>0\}}\frac{|\nabla f|_{*}^{q}}{f^{p-1}}d\mu<\infty
\]
and in this case we have 
\[
\mathsf{F}_{q}(f)=\int_{\{f>0\}}\frac{|\nabla f|_{*}^{q}}{f^{p-1}}d\mu.
\]
In addition, the functional is sequentially lower semicontinuous w.r.t.
the strong convergence in $L^{2r}(M,\mu)$ and $L^{2}(M,\mu)$. If
$p<2$ then the functional is also convex.\end{prop}
\begin{rem*}
Compare this to \cite[Remark 6.2]{Ambrosio2011} and \cite[Lemma 4.10]{Ambrosio2013}.
And note that the statement $|\nabla f|_{w}\in L^{1}$ follows already
from $f\in L^{1}$ and $\int\frac{|\nabla f|_{w}^{2}}{f}d\mu<\infty$
by applying the reverse Hölder inequality. \end{rem*}
\begin{proof}
Similar to \cite[Lemma 4.10]{Ambrosio2013} first assume $f$ is bounded.
Then note that $f\in D(\mathsf{F}_{q})$ requires $f^{r}\in L^{2}(M,\mu)$,
i.e. $f\in L^{2r}(M,\mu)$ and by chain rule 
\[
|\nabla f^{r}|_{*}^{q}=r^{q}\frac{|\nabla f|_{*}^{q}}{f^{p-1}}.
\]

Conversely, just use $\phi(r)=\sqrt{r+\epsilon}-\sqrt{\epsilon}$,
apply the chain rule and let $\epsilon\to0$. 

Convexity for $p<2$ follows from \cite{Borwein1997}: Since in that
case $q\ge p$, we know $(x,y)\mapsto x^{q}/y^{p-1}$ is convex in
$\mathbb{R}^{2}$.

\end{proof}

\subsection*{Absolutely continuous curves and geodesics}

If $I\subset\mathbb{R}$ is an open interval then we say that a curve
$\gamma:I\to X$ is in $AC^{p}(I,X)$ (we drop the metric $d$ for
simplicity) for some $p\in[1,\infty]$ if 
\[
d(\gamma_{s},\gamma_{t})\le\int_{s}^{t}g(r)dr\quad\forall s,t\in J:s<t
\]
for some $g\in L^{p}(J)$. In case $p=1$ we just say that $\gamma$
is absolutely continuous. It can be shown \cite[Theorem 1.1.2]{AmbGigSav2008}
that in this case the metric derivative 
\[
|\dot{\gamma}_{t}|:=\limsup_{s\to t}\frac{d(\gamma_{s},\gamma_{t})}{|s-t|}
\]
with $\lim$ for a.e. $t\in I$ is a minimal representative of such
a $g$. We will say $\gamma$ has constant (unit) speed if $|\dot{\gamma}_{t}|$
is constant (resp. $1$) almost everywhere in $I$.

It is not difficult to see that $AC^{p}(I,X)\subset C(\bar{I},X)$
where $C(\bar{I},X)$ is equipped with the $\sup$ distance $d^{*}$
\[
d^{*}(\gamma,\gamma'):=\sup_{t\in\bar{I}}d(\gamma_{t},\gamma_{t}').
\]
For each $t\in\bar{I}$ we can define the evaluation map $e_{t}:C(\bar{I},X)\to X$
by 
\[
e_{t}(\gamma)=\gamma_{t}.
\]

We will say that $(X,d)$ is a geodesic space if for each $x_{0},x_{1}\in X$
where is a constant speed curve $\gamma:[0,1]\to X$ with $\gamma_{i}=x_{i}$
and 
\[
d(\gamma_{s},\gamma_{t})=|t-s|d(\gamma_{0},\gamma_{1}).
\]
In this case, we say that $\gamma$ is a constant speed geodesic.
The space of all constant speed geodesics $\gamma:[0,1]\to X$ will
be donated by $\operatorname{Geo}(X)$. Using the triangle inequality
it is not difficult to show the following.
\begin{lem}
Assume $\gamma:[0,1]\to X$ is a curve such that 
\[
d(\gamma_{s},\gamma_{t})\le|t-s|d(\gamma_{0},\gamma_{1})
\]
then $\gamma$ is a geodesic from $\gamma_{0}$ to $\gamma_{1}$.
\end{lem}

A weaker concept is a length space: In such spaces the distance between
point $x_{0}$ and $x_{1}\in X$ is given by 
\[
d(x_{0},x_{1})=\inf\int_{0}^{1}|\dot{\gamma}_{t}|dt
\]
where the infimum is taken over all absolutely continuous curves connecting
$x_{0}$ and $x_{1}$. In case $X$ is complete and locally compact,
the two concepts agree. Furthermore, Arzela-Ascoli also implies:
\begin{lem}
If $(X,d)$ is locally compact then so is $(\operatorname{Geo}(X),d^{*})$
where $d^{*}$ is the $\sup$-distance on $C(\bar{I},X)$.
\end{lem}

\subsection*{Geodesically convex functionals and gradient flows}

A functional $E:X\to\mathbb{R}\cup\{+\infty\}$ is said to be $K$-geodesically
convex for some $K\in\mathbb{R}$ if for each $x_{0},x_{1}\in D(E)$
there is a geodesic $\gamma\in\operatorname{Geo}(X)$ connecting $x_{0}$
and $x_{1}$ such that 
\[
E(\gamma_{t})\le(1-t)E(\gamma_{0})+tE(\gamma_{1})-\frac{K}{2}(1-t)td^{2}(\gamma_{0},\gamma_{1}).
\]
In such a case it can be shown (\cite[Section 2.4]{AmbGigSav2008}
that the descending slope is an upper gradient of $E$ and can be
express as 
\[
|D^{-}E|(x)=\sup_{y\in X\backslash\{x\}}\left(\frac{E(x)-E(y)}{d(x,y)}+\frac{K}{2}d(x,y)\right)
\]
In particular, it is lower semicontinuous if $E$ is. Furthermore,
if $x:[0,\infty)\to D(E)$ is a locally absolutely continuous curve
then 
\[
E(x_{t})\ge E(x_{s})-\int_{s}^{t}|\dot{x}_{r}||D^{-}E|(y_{r})dr
\]
for every $s,t\in[0,\infty)$ and $s<t$. Note by Young's inequality
we also have for any $p\in(1,\infty)$
\[
E(x_{t})\ge E(x_{0})-\frac{1}{p}\int_{0}^{t}|\dot{x}_{t}|^{p}dt-\frac{1}{q}\int_{0}^{t}|D^{-}E|^{q}(x_{r})dr.
\]

\begin{defn}
[$(E,p)$-dissipation inequality and metric gradient flows] Let $E:X\to\mathbb{R}\cup\{\infty\}$
be a functional on $X$ then we say that a locally absolutely continuous
curve $t\mapsto y_{t}\in D(E)$ satisfies the $(E,p)$-dissipation
inequality if for all $t\ge0$ 
\[
E(x_{0})\ge E(x_{t})+\frac{1}{p}\int_{0}^{t}|\dot{x}_{t}|^{p}dt+\frac{1}{q}\int_{0}^{t}|D^{-}E|^{q}(x_{r})dr.
\]
$t\mapsto x_{t}$ is a gradient flow of $E$ starting at $y_{0}\in D(E)$
if 
\[
E(y_{0})=E(x_{t})+\frac{1}{p}\int_{0}^{t}|\dot{x}_{t}|^{p}dt+\frac{1}{q}\int_{0}^{t}|D^{-}E|^{q}(x_{r})dr.
\]

\end{defn}
In the geodesically convex case we immediately see that if $t\mapsto x_{t}$
satisfies the $(E,p)$-dissipation inequality then it is a (generalized)
gradient flow and 
\[
\frac{d}{dt}E(x_{t})=-|\dot{x}_{t}|^{p}=-|D^{-}E|^{q}(x_{t})
\]
for almost all $t\in(0,1)$.
\begin{rem*}
The theory developed in \cite{AmbGigSav2008} covers mainly the case
$p=2$ and only mentioned the required adjustments. For a comprehensive
treatment of the case $p\ne2$ and even more general situations see
\cite{Rossi2008}.
\end{rem*}

\subsection*{Wasserstein spaces}

In this section, we will give a short introduction to the Wasserstein
space $\mathcal{P}_{p}(M)$; for an overview of its general properties
see \cite[Chapter 6]{Villani2009}. 

Fix some $x_{0}\in M$ and let $\mathcal{P}(M)$ be the set of probability
measures on $M$. Denote by $\mathcal{P}_{p}(M)$ the following set
\[
\left\{\mu\in\mathcal{P}(M)\,|\,\int d^{p}(x,x_{0})d\mu(x)\right\}.
\]
It can be shown that the following object $w_{p}(\cdot,\cdot)$ defines
a complete metric on $\mathcal{P}_{p}(M)$.
\[
w_{p}^{p}(\mu_{0},\mu_{1})=\inf_{\pi\in\Pi(\mu_{0},\mu_{1})}\frac{1}{p}\int d^{p}(x,y)d\pi(x,y)
\]
 where $\Pi(\mu_{0},\mu_{1})$ is the set of all $\pi\in\mathcal{P}(M\times M)$
with $(p_{1})_{*}\pi=\mu_{0}$ and $(p_{2})_{*}\pi=\mu_{1}$ with
$p_{i}$ be the projections to the $i$-the coordinate. We will say
that $(\mathcal{P}_{p}(M),w_{p})$ is the $p$-Wasserstein space (modeled
on $(M,d)$).

Furthermore, it is well-known that $\mathcal{P}_{p}(M)$ is a geodesic/length
space if $M$ is. However, it is compact if and only if $M$ is. In
that case it agrees with the space of probability measures and the
topology induced by $w_{p}$ agrees with the weak topology on $\mathcal{P}(M)$.
Nevertheless, we have the following nice property:
\begin{lem}
[{\cite[Theorem 6]{Kell2011}}]\label{lem:weak-cpt-wasserstein}Let
$(M,d)$ be a proper metric space, then every bounded set in $\mathcal{P}_{p}(M)$
is precompact w.r.t. to the weak topology induced by $\mathcal{P}_{p}(M)\subset\mathcal{P}(M)$.\end{lem}
\begin{proof}
Let $x_{0}$ be some fixed measure in $M$. By \cite[Lemma 4.3]{Villani2009}
we know that the $w_{p}(\delta_{x_{0}},\cdot)$ is lower semicontinuous
w.r.t. to the weak convergence of measures. Thus we only need to prove
tightness of every $w_{p}$-ball $B_{R}^{p}(\delta_{x_{0}})\subset\mathcal{P}_{p}(M)$,
i.e. for every $\epsilon>0$ there is a compact set $K_{\epsilon}\subset M$
such that every $\mu\in B_{R}^{p}(\delta_{x_{0}})$
\[
\mu(M\backslash K_{\epsilon})\le\epsilon.
\]

If we set $K_{\epsilon}=B_{\frac{1}{\epsilon}}(x_{0})$ then 
\begin{eqnarray*}
\mu(M\backslash K_{\epsilon}) & \le & \epsilon^{p}\int_{d(x,x_{0})\ge\frac{1}{\epsilon}}d^{p}(x,x_{0})d\mu(x)\\
 & \le & \epsilon^{p}pw_{p}^{p}(\delta_{x_{0}},\mu)\le\epsilon pR^{p}
\end{eqnarray*}
which implies tightness since any ball in $M$ is compact.
\end{proof}
We say that a function $E:\mathcal{P}_{p}(M)\to\mathbb{R}\cup\{\infty\}$
is weakly lower semi-continuous, if it is lower semicontinuous w.r.t.
the weak topology on $\mathcal{P}_{p}(M)\subset\mathcal{P}(M)$. In
particular, the weak closure of bounded subset of sublevels of $E$
are contained in that sublevel.
\begin{thm}
Let $(M,d)$ be a proper geodesic metric space and $E$ be a functional
on $\mathcal{P}_{p}(M)$ such that $E$ and $|D^{-}E|$ are weakly
lower semicontinuous. Then for all $\mu_{0}\in D(E)$ there exists
a gradient flow $t\mapsto\mu_{t}$ of $E$ starting at $\mu_{0}$.\end{thm}
\begin{proof}
Just note by the previous lemma the assumptions \cite[Assumption 2.4a,c]{AmbGigSav2008}
hold and thus \cite[Corollary 2.4.12]{AmbGigSav2008} can be applied.\end{proof}
\begin{rem*}
The requirement $|D^{-}E|$ to be weakly lower semicontinuous is rather
restrictive in the non-compact case. Note, however, below we only
need lower semicontinuity, which follows from $K$-convexity. Existence
will follow from existence of the $q$-heat equation.
\end{rem*}

\section{The functional}

Given a function $U\in\mathcal{DC}_{N}$ for $N\in[1,\infty]$ we
write \textbf{$U'(\infty)=\lim_{r\to\infty}\frac{U(r)}{r}$}. Let
$\mu\in\mathcal{P}(M)$ be some reference measure, we define the functional
$\mathcal{U}_{\mu}:\mathcal{P}(M)\to\mathbb{R}\cup\{\infty\}$ by
\[
\mathcal{U}_{\mu}(\nu)=\int U(\rho)d\mu+U'(\infty)\nu_{s}(M)
\]
where $\nu=\rho\mu+\mu_{s}$ the the Lebesgue decomposition of $\nu$
w.r.t. $\mu$. 

In the following we usually fix a metric measure space $(M,d,\mu)$
and drop the subscript $\mu$ from the functional $\mathcal{U}_{\mu}$.
In addition, we use $\mathcal{U}_{m}$, $\mathcal{U}_{\alpha}$ etc.
to denote the functional generated by $U_{m}$, $U_{\alpha}$, etc.

Now let
\[
U_{p}(x)=\frac{1}{(3-p)(2-p)}(x^{3-p}-x)
\]
and let $\mathcal{U}_{p}$ be the associated functional. 
\begin{rem*}
The linear term in $U_{p}$ is just for cosmetic reasons, it does
not have any influence: Take $U=c\cdot x$ with $c>0$ and let $\mathcal{U}$
be the associated functional, then $U'(\infty)=c$ for $p\in(2,3)$
and thus 
\[
\mathcal{U}(\nu)=c\int fd\mu+c\cdot\nu^{s}(M)=c
\]
where $\nu=\rho\mu+\nu^{s}$ is the Lebesgue decomposition w.r.t.
$\mu$. Therefore, we have $\mathcal{U}_{p}(\nu)=\tilde{\mathcal{U}}_{p}(\nu)-\frac{1}{(3-p)(2-p)}$
with $\tilde{U}_{p}(x)=\frac{1}{(3-p)(2-p)}x^{3-p}$. For $p\in(1,2)$
we have $\mathcal{U}_{p}(\nu)<\infty$ iff $\nu^{s}=0$ and hence
the linear term is constant as well. 
\end{rem*}
Following the strategy in \cite[Section 7.2 and 8]{Ambrosio2013}
we will show that under a curvature conditions the $q$-heat flow
can be identified with the gradient flow of the function $\mathcal{U}_{p}$
in the $p$-Wasserstein space: More precisely, if $p\in(1,2)$ then
$3-p\in(1,2)$ and the functional is displacement convex if the strong
version of $CD_{p}(K,\infty)$ holds for some $K\ge0$ (see \cite{Kell2013a}
for definition of $CD_{p}(K,N)$). If $p\in(2,3)$ we have $3-p\in(0,1)$
so that $\mathcal{U}_{p}$ is displacement convex if $CD_{p}(0,N)$
holds with $1-\frac{1}{N}=3-p$. 
\begin{rem*}
Note, that in contrast to the case $p=2$, the strong version of $CD_{p}(K,\infty)$
does not imply $K$-convexity of functionals in $\mathcal{DC}_{\infty}$
for $K<0$ and $p<2$. We get $K'$-convexity in those cases if the
space is bounded (see \cite{Kell2013a}). Also Ohta and Takatsu could
show that on a weighted Riemannian manifold of non-negative Ricci
curvature the functional $\mathcal{U}_{p}$ with $p\in(2,3)$ is $K$-convex
in $\mathcal{P}_{2}(M)$ if $CD(K,N)$ holds (see \cite[Theorem 4.1]{OT2011}).
This is, however, not enough for $p\ne2$.
\end{rem*}
Recall from the introduction that $r>0$ will be an abbreviation for
$q\in(\frac{1+\sqrt{5}}{2},\infty)$, or equivalently $p\in(1,\frac{3+\sqrt{5}}{2})$.
\begin{lem}
Assume $r>0$ then 
\[
\nu\mapsto\mathcal{U}_{p}(\nu)
\]
is lower semicontinuous in $\mathcal{P}_{p}(M)$.\end{lem}
\begin{proof}
Just note that $U_{p}$ is convex and for $r>0$ we have $p\in(1,\frac{3+\sqrt{5}}{2})\subset(1,3)$
and thus $3-p>0$.\end{proof}
\begin{rem*}
The functional $\mathcal{U}_{p}$ appeared in a similar form already
in \cite[Proof of Lemma 3.13]{Gigli2012} and Otto's preprint \cite{Otto1996}
and also Augeh's thesis \cite{Agueh2002,Agueh2005}. Gigli used the
functional and the gradient flow of the $q$-Cheeger energy to show
that all gradients of $q$-Sobolev functions can be weakly represented
by a plan. In the Euclidean case, Otto and Augeh showed that the parabolic
$q$-Laplace equation, which is the $q$-heat flow for smooth solutions,
can be solved using the gradient flow of $\mathcal{U}_{p}$ in the
$p$-Wasserstein case. This should also be compared to \cite{OT2011,OT2011a},
where the (parabolic) porous media equation is solved via a gradient
flow of a similar functional in the $2$-Wasserstein space for Riemannian
manifolds of non-negative Ricci curvature. Note, however, no identification
is done. Furthermore, our approach shows that the abstract solution
of the $q$-heat flow solves the gradient flow problem in the $p$-Wasserstein
space.
\end{rem*}

\section{Gradient flow of the Cheeger energy in $L^{2}$}

We assume now that $\operatorname{Ch}_{q}$ is the $q$-Cheeger energy
on $(M,d,\mu)$ where $(M,d)$ is a proper metric space and $\mu$
is a $\sigma$-finite measure. From \cite[Proposition 4.1]{Ambrosio2013}
we know that the domain of $\operatorname{Ch}_{q}$ is dense in $L^{2}(M,\mu)$. 

Since $L^{2}(M,\mu)$ is Hilbert and $\operatorname{Ch}_{q}$ is convex
and lower semicontinuous, we can apply the classical theory of gradient
flows developed in \cite{Brezis1973} (see also \cite{AmbGigSav2008}).
For that recall that the subdifferential $\partial^{-}\operatorname{Ch}_{q}$
at $f\in D(\operatorname{Ch}_{q})$ is defined as (possibly empty)
set of functions $\ell\in L^{2}(M,\mu)$ such that for all $g\in L^{2}(M,\mu)$
\[
\int\ell(g-f)d\mu\le\operatorname{Ch}_{q}(g)-\operatorname{Ch}_{q}(f).
\]
If $f\notin D(\operatorname{Ch}_{q})$ then $\partial\operatorname{Ch}_{q}(f)=\varnothing$.
The domain $D(\partial\operatorname{Ch}_{q})$ of $\partial^{-}\operatorname{Ch}_{q}$
will be all $f\in L^{2}(M,\mu)$ such that $\partial^{-}\operatorname{Ch}_{q}\ne\varnothing$,
which is dense in $L^{2}(M,\mu)$ (see \cite[Proposition 2.11]{Brezis1973}).

By \cite{Brezis1973} the gradient flow of $\operatorname{Ch}_{q}$
gives for all $f_{0}\in L^{2}(M,\mu)$ a locally Lipschitz map $t\mapsto f_{t}=H_{t}(f_{t})$
(we drop $q$ if no confusion arises) from $(0,\infty)$ to $L^{2}(M,\mu)$
and $f_{t}\to f_{0}$ in $L^{2}(M,\mu)$ as $t\to0$ and the derivative
satisfies 
\[
\frac{d}{dt}f_{t}\in-\partial^{-}\operatorname{Ch}_{q}(f_{t})\qquad\mbox{for a.e. }t\in(0,\infty).
\]

\begin{defn}
[$q$-Laplacian] Let $f\in D(\partial\operatorname{Ch}_{q})$ then
$\Delta_{q}f$ is defined as the element $\ell\in-\partial^{-}\operatorname{Ch}_{q}(f)$
of minimal $L^{2}$-norm.
\end{defn}
By \cite[Theorem 3.2]{Brezis1973} we have the regularization effect
that $\frac{d^{+}}{dt}f_{t}$ exists everywhere in $(0,\infty)$ and
is the element $\ell\in-\partial^{-}\operatorname{Ch}_{q}(f_{t})$
with minimal $L^{2}$-norm, i.e. $\frac{d^{+}}{dt}f_{t}=\Delta_{q}f_{t}$.
\begin{rem*}
We can also define the $L$-Laplacian using the same theory where
$L$ is a convex increasing function with $L(0)=0$. Since such flows
might be interesting in combination with Orlicz-Wasserstein spaces,
we will analyze these flows in the future.\end{rem*}
\begin{prop}
[Properties of the Laplacian] If $f\in D(\Delta_{q})$ and $g\in D(\operatorname{Ch}_{q})$
then 
\[
-\int g\Delta_{q}fd\mu\le\int|\nabla g|_{*}|\nabla f|_{*}^{q-1}d\mu.
\]
Equality holds if $g=\phi(f)$ for some Lipschitz function $\phi:J\to\mathbb{R}$
with $J$ a closed interval containing the image of $f$ (and $\phi(0)=0$
if $\mu(M)=\infty$). In that case one also has
\[
-\int\phi(f)\Delta_{q}fd\mu=\int\phi'(f)|\nabla f|_{*}^{q}d\mu.
\]
If, in addition, $g\in D(\Delta_{q})$ and $\phi$ is nondecreasing
and Lipschitz on $\mathbb{R}$ with $\phi(0)=0$ then
\[
\int(\Delta_{q}g-\Delta_{q}f)\phi(g-f)d\mu\le0.
\]
\end{prop}
\begin{proof}
The first two parts were already proven in \cite[Proposition 6.5]{Ambrosio2011}
for $C^{1}$-functions $\phi$. However, using the proof of \cite[Proposition 4.15]{Ambrosio2013},
adapted to $p\ne2$, it can be proven in the same way. For convenience
we include the full proof: Since $-\Delta_{q}f\in\partial^{-}\operatorname{Ch}_{q}(f)$
we have for all $\epsilon>0$
\[
\operatorname{Ch}_{q}(f)-\int\epsilon g\Delta_{q}fd\mu\le\operatorname{Ch}_{q}(f+\epsilon g).
\]
Furthermore, $|\nabla f|_{*}+\epsilon|\nabla g|_{*}$ is a relaxed
slope of $f+\epsilon g$, we get 
\begin{eqnarray*}
-\int\epsilon g\Delta f & \le & \frac{1}{q}\int\left(|\nabla f|_{*}+\epsilon|\nabla g|_{*}\right)^{q}-|\nabla f|_{*}^{q}d\mu\\
 & = & \epsilon\int|\nabla g|_{*}|\nabla f|_{*}^{q-1}d\mu+o(\epsilon).
\end{eqnarray*}
Dividing by $\epsilon$ and letting $\epsilon\to0$ we obtain the
result.

In case $g=\phi(f)$ we apply chain rule and get $|\nabla(f+\epsilon\phi(f))|_{*}=(1+\epsilon\phi'(f))|\nabla f|_{*}$
and thus
\begin{eqnarray*}
\operatorname{Ch}_{q}(f+\epsilon\phi(f))-\operatorname{Ch}_{q}(f) & = & \frac{1}{q}\int|\nabla f|_{*}^{q}((1+\epsilon\phi(f))^{q}-1)d\mu\\
 & = & \epsilon\int\phi'(f)|\nabla f|_{*}^{q}d\mu+o(\epsilon).
\end{eqnarray*}

For the third part, just set $h=\phi(g-f)$, then $h\in D(\operatorname{Ch}_{q})$
and for $\epsilon>0$ 
\begin{eqnarray*}
-\epsilon\int(\Delta_{q}f-\Delta_{q}g)hd\mu & = & -\epsilon\int\Delta_{q}f\cdot hd\mu-\epsilon\int\Delta_{q}g\cdot(-h)d\mu\\
 & \le & \operatorname{Ch}_{q}(f+\epsilon h)-\operatorname{Ch}_{q}(f)+\operatorname{Ch}_{q}(g-\epsilon h)-\operatorname{Ch}_{q}(g).
\end{eqnarray*}
Taking $\epsilon$ sufficiently small such that $\epsilon\phi$ is
a contraction, we can apply Proposition \ref{prop:grad-contr} and
conclude.
\end{proof}
Actually with the help of Proposition \ref{prop:grad-contr} we can
also prove:
\begin{prop}
If $f,g\in D(\Delta_{L})$ and $\phi$ is nondecreasing and Lipschitz
on $\mathbb{R}$ with $\phi(0)=0$ then
\[
\int(\Delta_{L}g-\Delta_{L}f)\phi(g-f)d\mu\le0.
\]
\end{prop}
\begin{proof}
As above set $h=\phi(g-f)$, then $h\in D(\operatorname{Ch}_{q})$
and for $\epsilon>0$ 
\begin{eqnarray*}
-\epsilon\int(\Delta_{L}f-\Delta_{L}g)hd\mu & = & -\epsilon\int\Delta_{L}f\cdot hd\mu-\epsilon\int\Delta_{L}g\cdot(-h)d\mu\\
 & \le & \operatorname{Ch}_{L}(f+\epsilon h)-\operatorname{Ch}_{L}(f)+\operatorname{Ch}_{L}(g-\epsilon h)-\operatorname{Ch}_{L}(g).
\end{eqnarray*}
Then conclude by taking $\epsilon$ sufficiently small and applying
Proposition \ref{prop:grad-contr}.
\end{proof}
Using these results we can generalize \cite[Theorem 4.16]{Ambrosio2013}
to the case $p\ne2$ (and also \cite[Proposition 6.6]{Ambrosio2011}
where $0<c\le f_{0}\le C<\infty$ is required).
\begin{thm}
[Comparision principle and contraction] \label{thm:laplace-calculus}Let
$f_{t}=H_{t}(f_{0})$ and $g_{t}=H_{t}(g_{0})$ be the gradient flows
of $\operatorname{Ch}_{q}$ starting from $f_{0},g_{0}\in L^{2}(M,\mu)$
respectively. Then the following holds:
\begin{enumerate}
\item (Comparison principle) Assume $f_{0}\le C$ (resp. $f_{0}\ge c$).
Then $f_{t}\le C$ (resp. $f_{t}\ge c$) for every $t\ge0$. Similarly,
if $f_{0}\le g_{0}+C$ for some constant $C\in\mathbb{R}$, then $f_{t}\le g_{t}+C$.
\item (Contraction) If $e:\mathbb{R}\to[-l,\infty]$ is a convex lower semicontinuous
function and $E(f)=\int e(f)d\mu$ is the associated convex and lower
semicontinuous functional in $L^{2}(M,\mu)$ then
\[
E(f_{t})\le E(f_{0})\qquad\mbox{ for every }t\ge0,
\]
and 
\[
E(f_{t}-g_{t})\le E(f_{0}-g_{0})\qquad\mbox{ for every }t\ge0.
\]
In particular, $H_{t}:L^{2}(M,\mu)\to L^{2}(M,\mu)$ is a contraction
on $L^{2}(M,\mu)\cap L^{r}(M,\mu)$ w.r.t. the $L^{r}(M,\mu)$-norm
for all $r\ge1$, i.e. for all $f_{0},g_{0}\in L^{2}(M,\mu)\cap L^{r}(M,\mu)$
then
\[
\|H_{t}(f_{0})-H(g_{0})\|_{r}\le\|f_{0}-g_{0}\|_{r}.
\]

\item If $e:\mathbb{R}\to[0,\infty]$ is locally Lipschitz in $\mathbb{R}$
and $E(f_{0})<\infty$ then
\[
E(f_{t})+\int_{0}^{t}\int e''(f_{t})|\nabla f_{t}|_{*}^{q}d\mu ds=E(f_{0})\;\forall t\ge0.
\]

\item When $\mu(M)<\infty$ we have 
\[
\int f_{t}d\mu=\int f_{0}d\mu\qquad\mbox{ for every }t\ge0.
\]

\end{enumerate}
\end{thm}
\begin{rem*}
(1) The first two assertions also hold for the gradient flow of the
$L$-Cheeger energy, we will leave the details to the reader.\end{rem*}
\begin{proof}
The proof follows along the lines of \cite[Theorem 4.16]{Ambrosio2013}.
We will only show the result assuming $e'$ is bounded and globally
Lipschitz. By the same approximation as in \cite[Theorem 4.16]{Ambrosio2013}
the result follows.

Note first that the first statement follows by choosing $e(r)=\max\{r-C,0\}$
(resp. $e(r)=\max\{c-r,0\}$). 

So let $e'$ be bounded and Lipschitz on $\mathbb{R}$ then for $x,y\in\mathbb{R}$
we have 
\[
|e'(x)|\le|e'(0)|+\operatorname{Lip}(e')|x|,
\]
\[
|e(y)-e(x)-e'(x)(y-x)|\le\frac{1}{2}\operatorname{Lip}(e')|y-x|^{2}
\]
\[
|e(y)-e(x)|\le\left(|e'(0)|+\operatorname{Lip}(e')\right)\left(|x|+|y-x|\right)|y-x|,
\]
where we assume $e'(0)=e(0)=0$ if $\mu(M)=\infty$. Furthermore,
we will assume w.l.o.g. $E(f_{0}-g_{0})<\infty$ (which forces $e(0)=0$
if $\mu(M)=\infty$).

By convexity of $\operatorname{Ch}_{q}$ the maps $t\mapsto f_{t}$
and $t\mapsto g_{t}$ are locally Lipschitz continuous in $(0,\infty)$
with values in $L^{2}(M,\mu)$ (see \cite[Theorem 2.4.15]{AmbGigSav2008}
and \cite[Theorem 3.2]{Brezis1973}). Thus, the map $t\mapsto e(f_{t}-g_{t})$
is locally Lipschitz in $(0,\infty)$ with values in $L^{1}(M,\mu)$,
in particular, wherever $t\mapsto f_{t}$ and $t\mapsto g_{t}$ are
commonly differentiable, we have
\begin{eqnarray*}
\frac{d}{dt}e(f_{t}-g_{t}) & = & e'(f_{t}-g_{t})\frac{d}{dt}(f_{t}-g_{t})\\
 & = & e'(f_{t}-g_{t})(\Delta_{q}f_{t}-\Delta_{q}g_{t})\le0.
\end{eqnarray*}
Hence the function is $t\mapsto E(f_{t}-g_{t})$ is locally Lipschitz
in $(0,\infty)$. Integrating we see that the second assertion holds.

For the third statement, set $g_{0}=g_{t}=0$. Absolute continuity
of $t\mapsto E(f_{t})$ and the previous theorem yields for $\phi=e'$
\[
\frac{d}{dt}\int e(f_{t})d\mu=\int e'(f_{t})\Delta_{q}f_{t}d\mu=-\int e''(f_{t})|\nabla f_{t}|_{*}^{q}d\mu.
\]
In case $\mu(M)<\infty$ we can choose $e(r)=r$ and thus 
\[
\frac{d}{dt}\int f_{t}d\mu=-\int0\cdot|\nabla f_{t}|_{*}^{q}d\mu
\]
and hence $\int f_{t}d\mu=\int f_{0}d\mu$.
\end{proof}
In order to prove mass preservation for $\mu(M)=\infty$ we adjust
\cite[Section 4.4]{Ambrosio2013}. First we recall some facts about
the $p$-logarithm (see also \cite[Section 3]{OT2011}) which will
make the notation below easier.
\begin{lem}
The following inequality holds for $p\in(2,3)$, $x\ge0$ and $V\ge0$
\begin{eqnarray*}
x\ln_{p}x & \ge & x-\exp_{p}(-V^{p})-(p-2)V^{p}\exp_{p}(-V^{p})\\
 &  & +(p-3)V^{p}x
\end{eqnarray*}
where 
\[
\exp_{p}(t)=\left\{ 1+(2-p)t\right\} ^{\frac{1}{2-p}}
\]
and 
\[
\ln_{p}(s)=\frac{s^{2-p}-1}{2-p}
\]
which are inverse of each other for $t\in(-\infty,\frac{1}{p-2}]$.
Note also that $\exp_{p}$ is monotone on its domain and for sufficiently
small $h$
\[
\exp_{p}(h)\cdot\exp_{p}(-h)\le2.
\]
\end{lem}
\begin{proof}
Note first that $x\ln_{p}x$ is convex and thus 
\begin{eqnarray*}
x\ln_{p}x & \ge & x_{0}\ln_{p}x_{0}+(\ln_{p}x_{0}+x_{0}^{2-p})(x-x_{0})\\
 & = & x\ln_{p}x_{0}+x_{0}^{2-p}(x-x_{0}).
\end{eqnarray*}
Now choosing $x_{0}=\exp_{p}(-V^{p})\ge0$ then 
\begin{eqnarray*}
x_{0}^{2-p}(x-x_{0}) & = & \left\{ 1+(p-2)V^{p}\right\} ^{\frac{2-p}{2-p}}(x-\exp_{p}(-V^{p}))\\
 & = & x-exp_{p}(-V^{p})-(p-2)V^{p}\exp_{p}(-V^{p})+(p-2)V^{p}x
\end{eqnarray*}
Since $x\ln_{p}x_{0}=-V^{p}x$ we see that 
\[
x\ln_{p}x\ge x-\exp_{p}(-V^{p})-(p-2)V^{p}\exp_{p}(-V^{p})+(p-3)V^{p}x.
\]

\end{proof}

\begin{lem}
[Momentum-entropy estimate] Assume $p\in(1,3).$ Let $\mbox{\ensuremath{\mu}}$
be a finite measure and $V:X\to[0,\infty)$ be a Lipschitz function
with $V\ge\epsilon>0$ such that 
\[
I_{p}:=\begin{cases}
0 & \mbox{if }p\in(1,2)\\
\frac{p-2}{3-p}\int V^{p}\exp_{p}(-V^{p})d\mu & \mbox{if }p\in(2,3)
\end{cases}
\]
is finite and if $p\in(2,3)$ assume in addition 
\[
\int\exp_{p}(-V^{p})d\mu\le1
\]

Let $f_{0}\in L^{2}(X,\mu)$ be non-negative with
\[
\int V^{p}f_{0}d\mu<\infty
\]
and for some $z>0$
\[
z\int\exp_{p}(-V^{p})d\mu\le\int f_{0}d\mu
\]
if $p\in(2,3)$ and otherwise choose $z\le1.$ Then $t\mapsto\int V^{p}f_{t}d\mu$
is locally absolutely continuous in $[0,\infty)$ and for every $t\ge0$
\[
\int V^{p}f_{t}d\mu\le S_{t}
\]
 and 
\[
\int_{0}^{t}\int_{\{f_{s}>0\}}\frac{|\nabla f_{s}|_{*}^{q}}{f_{s}^{p-1}}d\mu ds\le\frac{4}{3-p}S_{t}
\]
where 
\[
S_{t}=e^{C_{p}\operatorname{Lip}(V)^{q}t}\left(I_{p}+\int\frac{1}{(2-p)}(f_{0}^{3-p}-f_{0})+(pV^{p}+z^{-1}l_{p})f_{0}d\mu\right)
\]
with $C_{p}=(p\cdot(3-p)^{-1})^{q}/q$ and $l_{p}=\max\{\frac{1}{2-p},1\}$.\end{lem}
\begin{proof}
Define the following
\[
\begin{array}{ll}
M^{q}(t):=\int V^{p}f_{t}d\mu, & E(t):=\frac{1}{(3-p)(2-p)}\int f_{t}^{3-p}-f_{t}d\mu,\\
F^{p}(t):=\int_{\{f_{t}>0\}}\frac{|\nabla f_{t}|_{*}^{q}}{f_{t}^{p-1}}d\mu.
\end{array}
\]
Applying Theorem \ref{thm:laplace-calculus} (see remark below that
theorem) to $(f_{t}+\epsilon)=H_{t}(f_{t}+\epsilon)$ and letting
$\epsilon\to0$ we see that $F\in L^{p}(0,T)$ for every $T>0$ and
\[
\frac{d}{dt}E(t)=-F^{p}(t)\;\mbox{ a.e. in }(0,T).
\]
Furthermore, by the Lemma above, conservation of mass and the assumption
$\int\exp_{p}(-V^{p})d\mu\le1$, we have for $p\in(2,3)$ 
\begin{eqnarray*}
(3-p)E(t) & = & \int f_{t}\ln_{p}f_{t}d\mu\\
 & \ge & \int f_{t}-\exp_{p}(-V^{p})d\mu-(p-2)\int V^{p}\exp_{p}(-V^{p})d\mu\\
 &  & +(p-3)M^{q}(t)\\
 & \ge & \int f_{0}-\exp_{p}(-V^{p})d\mu-I_{p}-(p-1)M^{q}(t)\\
 & \ge & (1-z^{-1})\int f_{0}d\mu-I_{p}+(p-3)M^{q}(t)\\
 & \ge & -z^{-1}l_{p}\int f_{0}d\mu-I_{p}+(p-3)M^{q}(t)
\end{eqnarray*}
 For $p\in(1,2)$ note that $\frac{1}{(2-p)}x^{3-p}\ge0$ and hence
\begin{eqnarray*}
(3-p)E(t) & = & \frac{1}{(2-p)}\int f_{t}^{3-p}-f_{t}d\mu\\
 & \ge & -\frac{1}{(2-p)}\int f_{0}d\mu\\
 & \ge & -z^{-1}l_{p}\int f_{0}d\mu-I_{p}+(p-3)M^{q}(t).
\end{eqnarray*}

In order to estimate the derivative of $M(t)$ we introduce a truncated
weight $V_{k}(x)=\min\{V(x),k\}$ and the corresponding functional
$M_{k}^{q}(t)$ as above. We know that the function $t\mapsto M_{k}^{q}(t)$
is locally Lipschitz continuous and thus for a.e. $t>0$
\begin{eqnarray*}
\left|\frac{d}{dt}M_{k}^{q}(t)\right| & = & \left|\int V_{k}^{p}\Delta_{q}f_{t}d\mu\right|\\
 & \le & p\int V_{k}^{p-1}|\nabla V_{k}|_{*}|\nabla f_{t}|_{*}^{q-1}d\mu\\
 & \le & pL\int\left(V_{k}^{p-1}f_{t}^{\frac{1}{q}}\right)\cdot\left(\frac{|\nabla f_{t}|_{*}^{q-1}}{f_{t}^{\frac{1}{q}}}\right)d\mu\\
 & \le & pLF(t)M_{k}(t)
\end{eqnarray*}
using $\operatorname{Lip}V_{k}\le L$ and Hölder inequality (note
$(p-1)q=p$). 

Since by mass preservation $M_{k}(t)\ge\tilde{\epsilon}:=\epsilon\int f_{0}d\mu$,
we can apply Gronwall's inequality and get 
\[
M_{k}^{q}(t)\le M_{k}^{q}(0)\exp\left(\int_{0}^{t}\frac{pLF(s)}{M_{k}^{q-1}(s)}ds\right)\le M^{q}(0)\exp\left(\int_{0}^{t}\frac{pLF(s)}{\tilde{\epsilon}{}^{q-1}}ds\right)
\]
for $t\in[0,N]$. Thus $M_{k}^{q}(t)$ is uniformly bounded and by
monotone convergence, we obtain the same differential inequality for
$M^{q}(t)$, i.e. for $t\in[0,\infty)$
\[
\left|\frac{d}{dt}M^{q}(t)\right|=pLF(t)M(t).
\]

Now combining this with the result above we get 
\[
\frac{d}{dt}((3-p)E+pM^{q})+(3-p)F^{p}\le p^{2}LFM\le(3-p)F^{p}+C_{p}L^{q}M^{q}
\]
where 
\[
C_{p}=(p\cdot(3-p)^{-1})^{q}/q.
\]

Combining this with the inequality above, we get by the Gronwall inequality
\begin{eqnarray*}
-z^{-1}l_{p}\int f_{0}d\mu-I_{p}+(2p-3)M^{q}(t) & \le & (3-p)E(t)+pM^{q}(t)\\
 & \le & e^{C_{p}L^{q}t}\left((3-p)E(0)+pM^{q}(0)\right).\\
 & \le & e^{C_{p}L^{q}t}\bigg((3-p)E(0)+pM^{q}(0)\\
 &  & \qquad\qquad+z^{-1}l_{p}\int f_{0}d\mu+I_{p}\bigg).
\end{eqnarray*}
Furthermore, we have
\[
(3-p)\int_{0}^{t}F^{p}(s)ds\le(3-p)(E(0)-E(t))\le(3-p)E(0)+I_{q}+z^{-1}l_{p}\int f_{0}d\mu+(3-p)M^{q}(t).
\]

\end{proof}
Having established this, similar to \cite[Theorem 4.20]{Ambrosio2013}
we can show that the gradient flow of the $q$-Cheeger energy is mass
preserving even if the measure $\mu$ is just $\sigma$-finite. The
proof relies on an approximation procedure developed in \cite[Section 4.3]{Ambrosio2013}.
We will freely use the concepts and results during the proof. The
reader may consult \cite[Section 4.3]{Ambrosio2013} for further reference.
\begin{thm}
\label{thm:q-mass-pres}Assume $p\in(1,\infty)$. If $\mu$ is a $\sigma$-finite
measure such that for some Lipschitz function $V:X\to[\epsilon,\infty]$
for some $\epsilon>0$ such that for $p\in(2,3)$ 
\[
\int\exp_{p}(-V^{p})d\mu\le1
\]
and 
\[
\int V^{p}\exp_{p}(-V^{p})d\mu<\infty
\]
and for $p\in(1,2)$ there is an increasing function $\Phi:\mathbb{R}\to[0,\infty]$
such that 
\[
\int\Phi(-V^{p})d\mu\le1.
\]
 Then the gradient flow $H_{t}$ of the $q$-Cheeger energy is mass
preserving, i.e. for $f_{t}=H_{t}(f_{0})$ with $\int f_{0}d\mu<\infty$
\[
\int f_{t}d\mu=\int f_{0}d\mu.
\]
Moreover, if $f_{0}\in L^{2}(M,\mu)$ is nonnegative and 
\[
\int V^{p}f_{0}d\mu,\int f_{0}d\mu<\infty
\]
then the bound of the previous Lemma hold.\end{thm}
\begin{proof}
We will use the construction of \cite[Theorem 4.20]{Ambrosio2013},
see in particular, \cite[Proposition 4.17]{Ambrosio2013}. By homogeneity
of the $H_{t}$, i.e. $H_{t}\lambda f=\lambda^{q}H_{t}f$, we can
assume $\int f_{0}d\mu\le1$ if $\int f_{0}d\mu<\infty$. In case
$\int f_{0}d\mu=\infty$, we can find a sequence $f_{0}^{n}\le f_{0}$
such that $n\le\int f_{0}^{n}d\mu<\infty$. Since mass preservation
holds for those functions, we can use the comparison principle to
show that $\int f_{t}d\mu\ge n$ for all $n$ and hence it also holds
in the case $\int f_{0}d\mu=\infty$. So w.l.o.g. $\int f_{0}d\mu\le1$.

We first show the case $p\in(2,3)$. For that use the following approximation:
$\mu^{0}:=\exp_{p}(-V^{p})\mu$ and $\mu^{k}:=\exp_{p}(-V_{k}^{p})\mu^{0}=$
for $V_{k}:=\min(V,k)$. Then $\mu^{k}$ is an increasing family of
finite measures and 
\[
\lim\mu^{k}(B)=\mu(B)\quad\forall B\in\mathcal{B}(M).
\]
 Since $V$ is Lipschitz we see that the density of $\mu$ w.r.t.
$\mu^{0}$ is bounded from below and above on any bounded set. 

For each $\mu^{k}$ let $f_{t}^{k}=H_{t}^{k}(f_{0})$ be the gradient
flow starting at $f_{0}$. Then since $\int\exp_{p}(-V^{p})d\mu^{k}\le1$
we can apply the previous lemma with $z_{k}=\int f_{t}^{k}d\mu^{k}$
for all $t\ge0$ and obtain
\[
\int V^{p}f_{t}^{k}d\mu^{k}\le e^{2\operatorname{Lip}(V)^{q}t}\left(I_{p}+\int\frac{1}{(2-p)}(f_{0}^{3-p}-f_{0})+(pV^{p}+z^{-1}l_{p})f_{0}d\mu^{k}\right).
\]
Since $f_{t}^{k}\to f_{t}$ strongly in $L^{2}(X,\mu^{0})$ (see \cite[Proposition 4.17]{Ambrosio2013})
we can assume up to changing to a subsequence $f_{t}^{k}\to f_{t}$
$\mu$-almost every where, and thus Fatou's lemma and monotonicity
of $\mu^{k}$ implies
\[
\int V^{p}f_{t}d\mu\le\liminf_{k\to\infty}\int V^{p}f_{t}^{k}d\mu^{k}
\]
 and the bound of the previous lemma holds since $z_{k}\nearrow z=\int f_{0}d\mu=1$.

Now consider $A_{h}=\{x\in M\,|\, V(x)\le h\}$. Since we assume $\int\exp_{p}(-V^{p})d\mu\le1$
we can choose $h$ such that $\exp_{p}(h)\exp_{p}(-h)\le2$ and get
by monotonicity 
\[
\mu(A_{h})\le\int2\exp_{p}(h^{p})\exp_{p}(-V^{p})d\mu\le2exp_{p}(h^{p})<\infty
\]
and thus by $(4.42)$ of \cite[Proposition 4.17]{Ambrosio2013}
\[
\int_{A_{h}}f_{t}d\mu=\lim_{k\to\infty}\int_{A_{h}}f_{t}^{k}d\mu^{k}.
\]
From the bound on the $p$-th moment we obtain for every $t>0$ a
constant $C>0$ such that 
\[
h^{q}\int_{X\backslash A_{h}}f_{t}^{k}d\mu^{k}\le C
\]
for every $h>0$ and hence
\begin{eqnarray*}
\int f_{t}d\mu & \ge & \int_{A_{h}}f_{t}d\mu=\lim_{k\to\infty}\int_{A_{h}}f_{t}^{k}d\mu^{k}\\
 & \ge & z-\limsup_{k\to\infty}\int_{X\backslash A_{h}}f_{t}^{k}d\mu^{k}\ge z-C/h^{p}.
\end{eqnarray*}
Since $h$ is arbitrary and the integral of $f_{t}$ does not exceed
$z$ we see that $\int f_{t}d\mu=z$. The second inequality of the
previous lemma follows by lower semicontinuity of the Cheeger energy
(see \ref{prop:Fisher}).

Mass preservation for signed initial data $f_{0}$ follows by the
same arguments as in \cite[Theorem 4.20]{Ambrosio2013}.

In order to treat the case $p\in(1,2)$ let $\Phi$ be increasing
such that $\int\Phi(-V)d\mu\le1$ and construct a monotone approximation
$\mu^{k}=\Phi(-V_{k})\mu^{0}$ and proceed as above.\end{proof}
\begin{rem*}
Let $p\in(2,3)$ if $p\to2$ then the condition 
\[
\int\exp_{p}(-V^{p})d\mu\le1
\]
converges to 
\[
\int\exp(-V^{2})d\mu\le1
\]
which is precisely the condition used in \cite[(4.2)]{Ambrosio2013}.
Note, however, it is stronger: Assuming $p\in(2,3)$ and $(p-2)V^{p}\ge1$
we have 
\begin{eqnarray*}
\exp_{p}(-V^{p}) & = & \left\{ 1+(2-p)(-V^{p})\right\} ^{\frac{1}{2-p}}\\
 & \le & \left\{ 2(p-2)V^{p}\right\} ^{\frac{1}{2-p}}\\
 & = & CV^{\frac{p}{2-p}}\ge C\exp(-V^{2})
\end{eqnarray*}
if $V$ is sufficiently large. In the Euclidean setting with $V(x)\approx\|x\|$
we get
\begin{eqnarray*}
\int_{\mathbb{R}^{n}\backslash B_{1}(0)}\exp_{p}(-V^{p})d\lambda & \approx & \int_{\mathbb{R}\backslash B_{1}(0)}\|x\|^{-\frac{p}{p-2}}d\lambda\\
 & \approx & \int_{1}^{\infty}r^{-\frac{p}{p-2}}r^{n-1}dr
\end{eqnarray*}
which is finite if $p<\frac{2n}{(n-1)}$, i.e. $q>\frac{2n}{n+1}$.
However, note that we currently need the more restrictive condition
\[
\int V^{p}\exp_{p}(-V^{p})d\mu\approx\int_{1}^{\infty}r^{p-\frac{p}{p-2}}r^{n-1}d\mu
\]
which is finite iff 
\[
p-\frac{p}{p-2}+n<1,
\]
i.e. $p<\frac{1}{2}\left(3-n+\sqrt{n^{2}+2n+9}\right)\approx2+\frac{2}{n}-\frac{2}{n^{2}}-\frac{2}{n^{3}}+\mathcal{O}(\frac{1}{n^{4}})$
as $n\to\infty$. 
\end{rem*}

\section{Gradient flow in the $p$-Wasserstein space}

Throughout this section we will assuming that the gradient flow of
the $q$-Cheeger energy is mass preserving, i.e. the conditions of
Theorem \ref{thm:q-mass-pres} hold. Furthermore, we assume that all
slopes of Lipschitz functions are equal almost everywhere, i.e. 
\[
|Df|=|D^{\pm}f|\quad\mu\mbox{-almost everwhere}.
\]
This condition holds if the space satisfies a local doubling and Poincar\'e
condition, in particular if $CD_{p}(K,N)$ holds with $N<\infty$.

Our motivation for the functional $\mathcal{U}_{p}$ and the identification
is the Kuwada lemma. It appeared the first time in \cite{Kuwada2010}
for $p=2$ and was extended by Ambrosio-Gigli-Savar\'e to $p\ne2$
for finite measures and $0<c\le f_{0}\le C<\infty$.
\begin{lem}
[Kuwada lemma] \label{lem:Kuwada}Let $f_{0}\in L^{q}(M,\mu)$ be
non-negative and $(f_{t})_{t\in[0,\infty)}$ be the gradient flow
of the $q$-Cheeger energy starting from $f_{0}$. Assume $\int f_{0}d\mu=1$.
Then the curve $t\mapsto d\mu_{t}=f_{t}d\mu$ is absolutely continuous
in $\mathcal{P}_{p}(M)$ and 
\[
|\dot{\mu}_{t}|^{p}\le\int\frac{|\nabla f_{t}|_{*}^{q}}{f_{t}^{p-1}}d\mu\qquad\mbox{ for almost every }t\in(0,1).
\]
\end{lem}
\begin{proof}
The proof follows from \cite[Lemma 7.2]{Ambrosio2011} using Theorem
\ref{thm:laplace-calculus} above the requirement $0<c\le f_{0}\le C<\infty$
can be easily dropped.\end{proof}
\begin{rem*}
Formally this lemma can be extended to cover $\partial_{t}f_{t}=\Delta\phi(f_{t})$,
which includes the porous media equation, $\phi(r)=c_{m}\cdot r^{m}$.
The theorems below hold with minor adjustments as well. However, since
a general existence theory of such equations on abstract metric spaces
is not available, an identification is difficult using our approach.
This is exactly why Ohta-Takatsu \cite{OT2011,OT2011a} can only use
the gradient flows in $\mathcal{P}_{2}$ to get a solution, but they
do not identify the two flows.\end{rem*}
\begin{prop}
\label{prop:abs-cont}Let $M$ be a proper metric measure space. In
case $p\in(2,3)$ assume, in addition, that $M$ is compact and $n$-Ahlfors
regular for $3-p>1-\frac{1}{n}$, i.e. $n(p-2)<1$. If $r>0$ and
$\mathcal{U}_{p}(\mu_{0})<\infty$, then $|D^{-}\mathcal{U}_{p}|(\mu_{0})<\infty$
implies $\mu_{0}$ is absolutely continuous w.r.t. $\mu$ and if there
is a sequence of absolutely continuous measure $\mu_{n}$ such that
$w_{p}(\mu_{0},\mu_{n})\to0$ and 
\[
|D^{-}\mathcal{U}_{p}|(\mu_{0})=\lim_{n\to0}\frac{\mathcal{U}_{p}(\mu_{0})-\mathcal{U}_{p}(\mu_{n})}{w_{p}(\mu_{0},\mu_{n})}.
\]
\end{prop}
\begin{rem*}
(1) The proof is extracted from \cite[Proof of Claim 7.7 and Remark 7.8]{OT2011}.
It is stated in the smooth setting but also works in the Ahlfors regular
case. The proof depends on the Ahlfors regularity to show that $\mathcal{U}_{p}(\hat{\mu}_{r})<\mathcal{U}_{p}(\mu_{0})$,
but it might be interesting to know if Ahlfors regularity is really
needed. 

(2) The only time where this proposition is needed is during the proof
of Theorem \ref{thm:grad-upper-bound} which is based on \cite[Theorem 7.5]{Ambrosio2013}.
In order to use the coupling technique and convexity absolute continuity
of $\mu_{n}$ is essential. \end{rem*}
\begin{proof}
Let $m=3-p$. In case $m>1$ the measures $\mu_{0}$ and $\mu_{n}$
must be absolutely continuous. So we are left to show the cases $0<m<1$. 

First assume $\mu_{0}$ has non-trivial singular part, i.e. $\mu_{0}=f_{0}\mu+\mu^{s}$
where $\mu^{s}$ and $\mu$ are mutually singular. Define for each
$r>0$ a measure $\hat{\mu}_{r}$ as follows
\[
d\hat{\mu}_{r}(x)=\rho_{r}(x)d\mu(x):=\left\{ f_{0}(x)+\int\frac{\chi_{B_{r}(y)}(x)}{\mu(B_{r}(y))}d\mu^{s}(y)\right\} d\mu(x).
\]
Then we have 
\begin{eqnarray*}
\int\rho_{r}(x)^{m}d\mu & = & \int\left[\int\left\{ \frac{f_{0}(x)}{\mu^{s}(M)}+\frac{\chi_{B_{r}(y)}(x)}{\mu(B_{r}(y))}\right\} d\mu^{s}(y)\right]^{m}d\mu(x)\\
 & \ge & \mu^{s}(M)^{m-1}\int\left[\int\left\{ \frac{f_{0}(x)}{\mu^{s}(M)}+\frac{\chi_{B_{r}(y)}(x)}{\mu(B_{r}(y))}\right\} ^{m}d\mu^{s}(y)\right]d\mu(x)\\
 & \ge & \mu^{s}(M)^{m-1}\int\left[\int_{M\backslash B_{r}(y)}\frac{f_{0}}{\mu^{s}(M)}d\mu+\int_{B_{r}(y)}\frac{1}{\mu(B_{r}(y))^{m}}d\mu\right]d\mu^{s}(y)\\
 & = & \int f_{0}^{m}d\mu-\mu^{s}(M)^{-1}\int\left(\int_{B_{r}(y)}f_{0}^{m}d\mu\right)d\mu^{s}(y)\\
 &  & +\mu^{s}(M)^{m-1}\int\mu(B_{r}(y))^{1-m}d\mu^{s}(y).
\end{eqnarray*}
Ahlfors regularity implies that for some $C,c>0$
\[
c\cdot r^{n}\le\mu(B_{r}(y))\le C\cdot r^{n}
\]
 and thus
\[
\mu^{s}(M)^{m-1}\int\mu(B_{r}(y))^{1-m}d\mu^{s}(y)\ge c^{1-m}\mu^{s}(M)^{m}\cdot r^{n(1-m)}.
\]
Furthermore, notice 
\begin{eqnarray*}
\int_{B_{r}(y)}f_{0}^{m}d\mu & \le & \left(\int_{B_{r}(y)}f_{0}d\mu\right)^{m}\left(\int_{B_{r}(y)}d\mu\right)^{1-m}\\
 & \le & \left(\int_{B_{r}(y)}f_{0}d\mu\right)^{m}C^{1-m}r^{n(1-m)}.
\end{eqnarray*}
Since $\lim_{r\to0}\sup_{y\in M}\int_{B_{r}(y)}f_{0}d\mu=0$ we see
that for sufficiently small $r>0$ (note $(m-1)<1$)
\[
\mathcal{U}_{m}(\hat{\mu}_{r})\le\mathcal{U}_{m}(\mu_{0})-\tilde{C}\mu^{s}(M)^{m}\cdot r^{n(1-m)}.
\]
 Furthermore, by our assumption
\[
n(1-m)=n(p-2)<1.
\]

To estimate $w_{p}(\mu_{0},\hat{\mu}_{r})$ note that the density
of $\hat{\mu}_{r}$ is defined as follows 
\[
\rho_{r}^{s}(x):=\int\frac{\chi_{B_{r}(y)}(x)}{\mu(B_{r}(y))}d\mu^{s}(y).
\]
Now choose the following coupling $\pi$ between $\mu^{s}$ and $\rho_{r}^{s}\mu$
\[
d\pi(x,y)=\int\frac{\chi_{B_{r}(z)}(x)}{\mu(B_{r}(z))}d\mu(x)d(\operatorname{Id}\times\operatorname{Id})_{*}\mu^{s}(z,y)
\]
Then 
\begin{eqnarray*}
w_{p}^{p}(\mu_{0},\hat{\mu}_{r}) & \le & \frac{1}{p}\int d^{p}(x,y)d\pi(x,y)\\
 & \le & \frac{1}{p}\int\int d^{p}(x,y)\frac{\chi_{B_{r}(z)}(x)}{\mu(B_{r}(z))}d\mu(x)d(\operatorname{Id}\times\operatorname{Id})_{*}\mu^{s}(z,y)\\
 & \le & \frac{1}{p}\int r^{p}d\mu^{s}
\end{eqnarray*}
 and thus
\[
w_{p}(\mu_{0},\hat{\mu}_{r})\le r\left(\frac{\mu^{s}(M)}{p}\right)^{\frac{1}{p}}.
\]

Combining these we get
\begin{eqnarray*}
|D^{-}\mathcal{U}_{m}|(\mu_{0}) & \ge & \limsup_{r\to0}\frac{\mathcal{U}_{N}(\mu_{0})-\mathcal{U}_{N}(\hat{\mu}_{r})}{w_{p}(\mu_{0},\hat{\mu}_{r})}\\
 & \ge & \limsup_{r\to0}\frac{\tilde{C}\mu^{s}(M)^{m}\cdot r^{n(1-m)}}{r\left(\frac{\mu^{s}(M)}{p}\right)^{\frac{1}{p}}}=\infty
\end{eqnarray*}
since $n(1-m)<1$. Which implies that $\mu_{0}$ must be absolutely
continuous.

For the second part, a similar argument works. Given $\mu_{n}$ we
can construct $\hat{\mu}_{n}^{r}$ similar to $\hat{\mu}_{r}$. The
estimates for $\mathcal{U}_{m}$ hold without any change. For the
rest just note
\begin{eqnarray*}
\frac{w_{p}(\mu_{0},\hat{\mu}_{n}^{r})}{w_{p}(\mu_{0},\mu_{n})} & \le & \frac{1}{w_{p}(\mu_{0},\mu_{n})}\left\{ w_{p}(\mu_{0},\mu_{n})+w_{p}(\mu_{n},\hat{\mu}_{n}^{r})\right\} \\
 & \le & 1+\frac{1}{w_{p}(\mu_{0},\mu_{n})}r\left(\frac{\mu^{s}(M)}{p}\right)^{\frac{1}{p}}.
\end{eqnarray*}
Thus 
\[
\frac{\mathcal{U}_{m}(\mu_{0})-\mathcal{U}_{m}(\hat{\mu}_{n}^{r})}{w_{p}(\mu_{0},\hat{\mu}_{n}^{r})}\ge\frac{\mathcal{U}_{m}(\mu_{0})-\mathcal{U}_{m}(\mu_{n})+\tilde{C}\mu^{s}(M)r^{n(1-m)}}{w_{p}(\mu_{0},\mu_{n})}\cdot\left(1+\frac{r\mu_{n}^{s}(M)^{\frac{1}{p}}}{p^{\frac{1}{p}}w_{p}(\mu_{0},\mu_{n})}\right)^{-1}.
\]
Now choosing $r_{n}$ such that $r_{n}/w_{p}(\mu_{0},\mu_{n})\to0$
we see that 
\[
\limsup_{n\to\infty}\frac{\mathcal{U}_{m}(\mu_{0})-\mathcal{U}_{m}(\hat{\mu}_{n}^{r_{n}})}{w_{p}(\mu_{0},\hat{\mu}_{n})}\ge\limsup\frac{\mathcal{U}_{m}(\mu_{0})-\mathcal{U}_{m}(\mu_{n})}{w_{p}(\mu_{0},\mu_{n})}
\]
By maximality of $\mu_{n}$ we see that this has to be an equality,
so up to extracting a subsequence we get 
\[
|D^{-}\mathcal{U}_{p}|(\mu_{0})=\lim_{n\to0}\frac{\mathcal{U}_{p}(\mu_{0})-\mathcal{U}_{p}(\hat{\mu}_{n}^{r_{n}})}{w_{p}(\mu_{0},\hat{\mu}_{n})}.
\]
\end{proof}
\begin{thm}
Assume $r>0$ and let $\mu_{0}\in D(\mathcal{U}_{p})$ with $|D^{-}\mathcal{U}_{p}|(\mu_{0})<\infty$.
Then $\mu_{0}=\rho\mu$, $\rho^{r}\in D(\operatorname{Ch}_{q})$ and
\[
r^{-q}\int|\nabla\rho^{r}|_{*}^{q}d\mu\le|D^{-}\mathcal{U}_{p}|^{q}(\mu_{0}).
\]
\end{thm}
\begin{proof}
We will follow the strategy of \cite[Theorem 7.4]{Ambrosio2013}.
First assume $\rho\in L^{2}(M,\mu)$ and let $(\rho_{t})_{t\in(0,\infty)}$
be the gradient flow of the $q$-Cheeger energy starting from $\rho$.
Let $\mu_{t}=\rho_{t}\mu$ then according to the definition of the
$q$-Fisher information we have by Lemma \ref{thm:laplace-calculus}
and \ref{lem:Kuwada}
\begin{eqnarray*}
\mathcal{U}_{p}(\mu_{0})-\mathcal{U}_{p}(\mu_{t}) & \ge & \frac{1}{q}\int_{0}^{t}\mathsf{F}_{q}(\rho_{s})ds+\frac{1}{p}\int_{0}^{t}|\dot{\mu}_{s}|^{p}d\mu\\
 & \ge & \frac{1}{q}\left(\frac{1}{t^{\frac{1}{p}}}\int_{0}^{t}\sqrt[q]{\mathsf{F}_{q}(\rho_{s})}ds\right)^{q}+\frac{1}{p}\left(\frac{1}{t^{\frac{1}{q}}}\int_{0}^{t}|\dot{\mu}_{s}|ds\right)^{p}\\
 & \ge & \frac{1}{t}\left(\int_{0}\sqrt[q]{\mathsf{F}_{q}(\rho_{s})}ds\right)w_{p}(\mu_{0},\mu_{t}).
\end{eqnarray*}
Thus dividing by $w_{p}(\mu_{0},\mu_{t})$ and letting $t\to0^{+}$
we get the result, since lower-semicontinuity of $\mathsf{F}_{q}$
implies
\[
\sqrt[q]{\mathsf{F}_{q}(\rho_{0})}\le\liminf_{t\to0^{+}}\frac{1}{t}\int_{0}^{t}\sqrt[q]{\mathsf{F}_{q}(\rho_{s})}ds.
\]

In case just $\mathcal{U}_{p}(\mu_{0})<\infty$ holds we prove the
result by approximation: Let $\rho^{n}=\min\{\rho,n\}$ and $(\rho_{t}^{n})$
be the corresponding gradient flow of the $q$-Cheeger energy. Using
the comparison principle we see that $\rho_{t}=\lim_{n\to\infty}\rho_{t}^{n}$
almost everywhere. Thus using the fact that $z_{n}=\int\rho^{n}d\mu=\int\rho_{t}^{n}d\mu$
we deduce that $\mu_{t}^{n}=\frac{1}{z_{n}}\rho_{t}^{n}\mu$ converges
to $\mu_{t}=\rho_{t}\mu$ in $\mathcal{P}_{p}(M)$. Now using the
lower semicontinuity properties of $\mathcal{U}_{p}$ we deduce
\[
\mathcal{U}_{p}(\mu_{0})-\mathcal{U}_{p}(\mu_{t})\ge\frac{1}{t}\left(\int_{0}^{t}\sqrt[q]{\mathsf{F}_{q}}ds\right)w_{p}(\mu_{0},\mu_{t})
\]
and conclude as above. \end{proof}
\begin{thm}
\label{thm:grad-upper-bound}Assume $\mu$ is finite and, in addition
if $p>2$, assume also that $(M,d,\mu)$ is as in Proposition \ref{prop:abs-cont}.
Let $\mu_{0}=\rho\mu\in D(\mathcal{U}_{p})$ and assume $\rho$ is
a bounded Lipschitz continuous map with $\rho\ge\epsilon$. Then 
\[
|D^{-}\mathcal{U}_{p}|^{q}(\mu_{0})\le\int\frac{|D\rho|^{q}}{\rho^{p-1}}d\mu=r^{-q}\int|D\rho^{r}|^{q}d\mu,
\]
where $|D\rho|(x)=\max\{|D^{+}\rho|(x),|D^{-}\rho|(x)\}$.\end{thm}
\begin{rem*}
For $p<2$, we have $2-p>0$ and the idea of \cite[Theorem 7.5]{Ambrosio2013}
can be followed in a similar way using the approximation function
$\Phi$ (see Theorem \ref{thm:q-mass-pres}) so that a similar version
to that theorem follows. For $p>2$, we have $2-p<0$, so that an
appropriate version requires further work. Note, however, that Proposition
\ref{prop:abs-cont} requires $M$ to be compact and hence $\mu$
to be finite.\end{rem*}
\begin{proof}
Recall that 
\begin{eqnarray*}
U_{p}(r) & = & \frac{1}{(3-p)(2-p)}(x^{3-p}-x)\\
\tilde{U}_{p}(r) & = & \frac{1}{(3-p)(2-p)}x^{3-p}.
\end{eqnarray*}

Define 
\[
L(x,y):=\begin{cases}
\frac{\left(\frac{1}{2-p}\rho^{2-p}(x)-\frac{1}{2-p}\rho^{2-p}(y)\right)^{+}}{d(x,y)} & \mbox{ if }x\ne y\\
\frac{|D\rho|}{\rho^{p-1}} & \mbox{ if }x=y.
\end{cases}
\]
Note that $L$ is measurable and for fixed $x\in M$ the map $y\mapsto L(x,y)$
is upper semicontinuous. Furthermore, since $\rho$ is Lipschitz and
$\epsilon\le\rho\le M$, $L$ is bounded.

Now take a sequence of absolutely continuous measures $\mu_{n}$ with
$w_{p}(\mu_{0},\mu_{n})\to0$ and 
\[
|D^{-}\mathcal{U}_{p}|(\mu_{0})=\lim_{n\to0}\frac{\mathcal{U}_{p}(\mu_{0})-\mathcal{U}_{p}(\mu_{n})}{w_{p}(\mu_{0},\mu_{n})}.
\]
Let $\rho_{n}$ be the density of $\mu_{n}$ w.r.t. $\mu$ and $\pi_{n}$
be some $c_{p}$-optimal transport plan of $(\mu_{0},\mu_{n})$. Because
$r\mapsto U_{p}(r)$ is convex we have 
\begin{eqnarray*}
\mathcal{U}_{p}(\mu_{0})-\mathcal{U}_{p}(\mu_{n}) & = & \int\left(U_{p}(\rho)-U_{p}(\rho_{n})\right)d\mu\le\int U_{p}^{'}(\rho)(\rho-\rho_{n})d\mu\\
 & = & \int U_{p}^{'}(\rho)d\mu_{0}-\int U_{p}^{'}(\rho)d\mu_{n}=\int\left(\tilde{U}_{p}^{'}(\rho(x))-\tilde{U}_{p}^{'}(\rho(y))\right)d\pi_{n}(x,y)\\
 & \le & \int L(x,y)d(x,y)d\pi_{n}(x,y)\le w_{p}(\mu_{0},\mu_{n})\left(\int L^{q}(x,y)d\pi_{n}(x,y)\right)^{1/q}\\
 & = & w_{p}(\mu_{0},\mu_{n})\left(\int\left(\int L^{q}(x,y)d\pi_{n,x}(y)\right)d\mu_{0}(x)\right)^{1/q}
\end{eqnarray*}
where $\pi_{n,x}$ is the disintegration of $\pi_{n}$ w.r.t. the
first marginal $\mu_{0}$ and $\tilde{U}_{p}(x)=\frac{1}{(3-p)(2-p)}x^{2-p}$.
Since $\int(\int d^{p}(x,y)d\pi_{n,x}(y))d\mu_{0}(x)\to0$ we can
assume w.l.o.g. that for $\mu_{0}$-a.e. $x\in M$
\[
\lim_{n\to\infty}\int d^{p}(x,y)d\pi_{n,x}(y)=0
\]
 and in particular
\[
\int_{M\backslash B_{r}(x)}L^{q}(x,y)d\pi_{n,x}(y)\to0
\]
 for all $r>0$. Furthermore, notice
\begin{eqnarray*}
\limsup_{n\to\infty}\int L^{q}(x,y)d\pi_{n,x}(y) & \le & \limsup_{n\to\infty}\int_{B_{r}(x)}L^{q}(x,y)d\pi_{n,x}(y)\\
 &  & +\limsup_{n\to\infty}\int_{M\backslash B_{r}(x)}L^{q}(x,y)d\pi_{n,x}(y)\\
 & \le & \limsup_{n\to\infty}\int_{B_{r}(x)}L^{q}(x,y)d\pi_{n,x}(y)\le\sup_{y\in B_{r}(x)}L^{q}(x,y).
\end{eqnarray*}
By upper semicontinuity of $L(x,\cdot)$ we immediately get $\limsup_{n}\int L^{q}(x,y)d\pi_{n,x}(y)\le L^{q}(x,x)$
for $\mu_{0}$-almost every $x\in M$. Since $L$ is bounded, we can
use Fatou's lemma and conclude
\begin{eqnarray*}
|D^{-}\mathcal{U}_{p}|(\mu_{0}) & = & \lim_{n\to0}\frac{\mathcal{U}_{p}(\mu_{0})-\mathcal{U}_{p}(\mu_{n})}{w_{p}(\mu_{0},\mu_{n})}\\
 & \le & \int\limsup_{n\to\infty}\left(\int L^{q}(x,y)d\pi_{n,x}(y)\right)^{1/q}d\mu_{0}(x)\\
 & \le & \left(\int L^{q}(x,x)d\mu_{0}(x)\right)\\
 & = & \left(\int\frac{|D\rho|^{q}}{\rho^{(p-1)q}}\rho d\mu\right)^{1/q}=\left(\int\frac{|D\rho|^{q}}{\rho^{p-1}}\rho d\mu\right)^{1/q}.
\end{eqnarray*}
\end{proof}
\begin{prop}
\label{prop:gradient-Fisher}If $|D^{-}\mathcal{U}_{p}|$ is sequentially
lower semicontinuous w.r.t. $\mathcal{P}_{p}(M)$ then 
\[
|D^{-}\mathcal{U}_{p}|^{q}(\mu_{0})=r^{-q}\int|\nabla\rho^{r}|_{*}^{q}d\mu\qquad\forall\mu_{0}=\rho\mu\in D(\mathcal{U}_{p}).
\]
\end{prop}
\begin{rem*}
In \cite[Theorem 7.6]{Ambrosio2013} Ambrosio-Gigli-Savar\'e proved
also that the converse holds for the entropy functional. We are not
able to prove the converse in case $2r>1$, i.e. $p>2$. \end{rem*}
\begin{proof}
By the above results we only need to show that $|D^{-}\mathcal{U}_{p}|(\mu_{0})\le r^{-q}\int|\nabla\rho^{r}|_{*}^{q}d\mu$.
First assume $\rho$ is bounded and find a sequence of measures $\mu_{n}\in\mathcal{P}_{p}(M)$
with Lipschitz densities $\rho_{n}$ bounded from below by $\frac{1}{n}$
converging in $L^{2r}(M)$ to $\rho$ (by compactness $\rho_{n}^{r}\to\rho^{r}$
in $L^{2})$ such that 
\[
\lim_{n\to\infty}\frac{1}{q}\int|\nabla\rho_{n}^{r}|_{*}^{q}d\mu=\operatorname{Ch}_{q}(\rho^{r}).
\]
Since $|\nabla\rho_{n}^{r}|_{w}=|D\rho_{n}^{r}|$ almost everywhere,
we see that 
\begin{eqnarray*}
|D^{-}\mathcal{U}_{p}|(\mu_{0}) & \le & \liminf_{n\to\infty}|D^{-}\mathcal{U}_{p}|(\mu_{n})\\
 & \le & \liminf r^{-q}\int|\nabla\rho_{n}^{r}|_{*}^{q}d\mu=r^{-q}\int|\nabla\rho^{r}|_{*}^{q}d\mu.
\end{eqnarray*}
In case $\rho$ is unbounded we can truncate $\rho$ without increasing
the $q$-Cheeger energy use the lower semicontinuity again to conclude
the result.\end{proof}
\begin{cor}
Assume one of the following holds:
\begin{itemize}
\item $p\in(1,2)$ and the strong $CD_{p}(K,\infty)$ condition holds for
some $K\ge0$
\item $p\in(2,\frac{3+\sqrt{5}}{2})$, the $CD_{p}(0,N)$ condition holds
such that $p=\frac{2N+1}{N}$ and $M$ is $n$-Ahlfors regular for
some $n<N$.
\end{itemize}
Then $|D^{-}\mathcal{U}_{p}|$ is lower semicontinuous and an upper
gradient of $\mathcal{U}_{p}$.\end{cor}
\begin{proof}
In case $p\in(1,2)$ note that $3-p\in(1,2)$ and thus $U_{p}\in\mathcal{DC}_{\infty}$.
In case $p\in(2,\frac{3+\sqrt{5}}{2})$ we have $3-p\in(0,1)$ and
thus $U_{p}\in\mathcal{DC}_{N}$ for $3-p=1-\frac{1}{N}$. In both
cases displacement convexity, i.e. $K$-convexity with $K=0$, follows.
Which implies that $|D^{-}\mathcal{U}_{p}|$ is lower semicontinuous
and an upper gradient of $\mathcal{U}_{p}$.
\end{proof}
The conclusion holds equally if $\mathcal{U}_{p}$ is just $K$-convex.
Since $K$-convexity neither follows from the strong $CD_{p}(K,\infty)$-condition
in case $p\in(1,2)$ nor from $CD_{p}(K,N)$, we use those conditions
to imply convexity. Nevertheless, we hope that it is possible to show
that $ $$|D^{-}\mathcal{U}_{p}|$ is lower semicontinuous and an
upper gradient of $\mathcal{U}_{p}$ if one of the curvature condition
holds.
\begin{thm}
[Uniqueness of the gradient flow of $\mathcal{U}_p$] \label{thm:gradunique}Let
$r>0$ and assume that $|D^{-}\mathcal{U}_{p}|^{q}$ is lower semicontinuous
and convex w.r.t. linear interpolation. Then for every $\mu_{0}\in\mathcal{P}_{p}(M)$
there exists at most one gradient flow of $\mathcal{U}_{p}$ starting
from $\mu_{0}$.\end{thm}
\begin{rem*}
By Lemma \ref{prop:Fisher} and \cite[Theorem 7.8]{Ambrosio2013}
convexity of $|D^{-}\mathcal{U}_{p}|^{q}$ holds if $p\le2\le q$.\end{rem*}
\begin{proof}
Assume that $(\mu_{t}^{1})$ and $(\mu_{t}^{2})$ are two distinct
gradient flows starting from $\mu_{0}$. Then we have for $i=1,2$
and all $T\ge0$ 
\begin{eqnarray*}
\mathcal{U}_{p}(\mu_{0}) & = & \mathcal{U}_{p}(\mu_{T}^{i})+\frac{1}{p}\int_{0}^{T}|\dot{\mu}_{t}^{i}|^{q}dt\\
 &  & +\frac{1}{q}\int_{0}^{T}|D^{-}\mathcal{U}_{p}|^{q}(\mu_{t}^{i})dt.
\end{eqnarray*}

Note that the curve $t\mapsto\mu_{t}=(\mu_{t}^{1}+\mu_{t}^{2})/2$
is absolutely continuous in $\mathcal{P}_{p}(M)$ and 
\[
|\dot{\mu}_{t}|^{p}\le\frac{|\dot{\mu}_{t}^{1}|^{p}+|\dot{\mu}_{t}^{2}|^{p}}{2}.
\]
Using the strict convexity of $\mathcal{U}_{p}$ and the convexity
of $|D^{-}\mathcal{U}_{p}|^{q}$ we conclude
\begin{eqnarray*}
\mathcal{U}_{p}(\mu_{0}) & > & \mathcal{U}_{p}(\mu_{T})+\frac{1}{p}\int_{0}^{T}|\dot{\mu}_{t}|^{q}dt\\
 &  & +\frac{1}{q}\int_{0}^{T}|D^{-}\mathcal{U}_{p}|^{q}(\mu_{t})dt\\
 & \ge & \mathcal{U}_{p}(\mu_{T})+\int_{0}^{T}|\dot{\mu}_{t}||D^{-}\mathcal{U}_{p}|(\mu_{t})dt
\end{eqnarray*}
But this is a contradiction to 
\[
\mathcal{U}_{p}(\mu_{t})\ge\mathcal{U}_{p}(\mu_{s})-\int_{s}^{t}|\dot{\mu}_{t}||D^{-}\mathcal{U}|(\mu_{t})dt
\]
for $s,t\in[0,\infty)$ (note $|D^{-}\mathcal{U}_{p}|$ is an upper
gradient).
\end{proof}
Finally we can identify the two flows. The theorem and its proof is
similar to \cite[Theorem 8.5]{Ambrosio2013}.
\begin{thm}
[Identification of the gradient flows] Let $r>0$ and assume that
$\mathcal{U}_{p}$ is $K$-convex in $\mathcal{P}_{p}(M)$. Then for
all $f_{0}\in L^{2}(M,\mu)$ such that $\mu_{0}=f_{0}\mu\in\mathcal{P}_{p}(M)$
the following is equivalent:
\begin{enumerate}
\item If $f_{t}$ is the gradient flow of $\operatorname{Ch}_{q}$ in $L^{2}(M,\mu)$
starting from $f_{0}$, then $\mu_{t}=f_{t}\mu$ is the gradient flow
of $\mathcal{U}_{p}$ in $\mathcal{P}_{p}(M)$ starting from $\mu_{0}$,
the map $t\mapsto\mathcal{U}_{p}(\mu_{t})$ is absolutely continuous
in $(0,\infty)$ and 
\[
-\frac{d}{dt}\mathcal{U}_{p}(\mu_{t})=|\dot{\mu}_{t}|^{p}=|D^{-}\mathcal{U}_{p}|^{q}\qquad\mbox{ for a.e.}t\in(0,\infty).
\]

\item Conversely, if we assume in addition that $|D^{-}\mathcal{U}_{p}|^{q}$
is convex w.r.t. linear interpolation, then whenever $\mu_{t}$ is
the gradient flow of $\mathcal{U}_{p}$ in $\mathcal{P}_{p}(M)$ starting
from $\mu_{0}$, then$\mu_{t}$ is absolutely continuous and its density
$f_{t}$ w.r.t. $\mu$ is the gradient flow of $\operatorname{Ch}_{q}$
in $L^{2}(M,\mu)$ starting from $f_{0}$. The same holds if the gradient
flow of $\mathcal{U}_{p}$ in $\mathcal{P}_{p}(M)$ starting $\mu_{0}$
is unique.
\end{enumerate}
\end{thm}
\begin{proof}
By $K$-convexity of $\mathcal{U}_{p}$ we know that $|D^{-}\mathcal{U}_{p}|$
is an upper gradient and 
\[
|D^{-}\mathcal{U}_{p}|^{q}(\rho\mu)=\mathsf{F}_{q}(\rho)
\]
thus by the Kuwada lemma we know that if $f_{t}$ is the gradient
flow of the $q$-Cheeger energy then 
\[
|\dot{\mu}_{t}|^{p}\le\int\frac{|\nabla f_{t}|^{q}}{f_{t}^{p-1}}d\mu=\mathsf{F}_{q}(f_{t})
\]
and 
\[
t\mapsto\mathcal{U}_{p}(\mu_{t})
\]
 is absolutely continuous with 
\[
\frac{d}{dt}\mathcal{U}_{p}(\mu_{t})=-\int\frac{|\nabla f_{t}|^{q}}{f_{t}^{p-1}}d\mu.
\]
Hence 
\[
\int\frac{|\nabla f_{t}|^{q}}{f_{t}^{p-1}}d\mu\ge\frac{1}{p}|\dot{\mu}_{t}|^{p}+\frac{1}{q}|D^{-}\mathcal{U}_{p}|^{p}
\]
so that $\mu_{t}$ satisfies the $\mathcal{U}_{p}$-dissipation inequality,
i.e.
\begin{eqnarray*}
\mathcal{U}_{p}(\mu_{0})-\mathcal{U}_{p}(\mu_{t}) & = & \int_{0}^{t}\int\frac{|\nabla f_{s}|^{q}}{f_{s}^{p-1}}d\mu ds\\
 & \ge & \frac{1}{p}\int_{0}^{t}|\dot{\mu}_{t}|^{p}ds+\frac{1}{q}\int_{0}^{t}|D^{-}\mathcal{U}_{p}|^{q}ds
\end{eqnarray*}
and $\mu_{t}$ is the gradient flow of $\mathcal{U}_{p}$ in $\mathcal{P}_{p}(M)$
starting at $\mu_{0}$. Absolute continuity of $t\mapsto\mathcal{U}_{p}(\mu_{t})$
in $(0,\infty)$ implies
\begin{eqnarray*}
\frac{d}{dt}\mathcal{U}_{p}(\mu_{t}) & = & -|\mu_{t}||D^{-}\mathcal{U}_{p}|\\
 & = & -|\mu_{t}|^{p}\\
 & = & -|D^{-}\mathcal{U}_{p}|^{q}.
\end{eqnarray*}

For the second part, assume that $t\mapsto\tilde{f}_{t}$ is the gradient
flow of the $q$-Cheeger energy starting at $f_{0}$. By the previous
part we know that $\tilde{\mu}_{t}=\tilde{f}_{t}\mu$ is also a gradient
flow of $\mathcal{U}_{p}$. Uniqueness (Theorem \ref{thm:gradunique}
above) implies that $\mu_{t}=\tilde{\mu}_{t}$ for all $t\ge0$.
\end{proof}
\bibliographystyle{amsalpha}
\bibliography{bib}

\end{document}